\documentclass[10pt]{amsart}

\usepackage[latin1]{inputenc}
\usepackage{mathptmx}
\usepackage{amsmath}
\usepackage{dsfont}
\usepackage{amsfonts}
\usepackage{amssymb}
\usepackage{amsthm}
\usepackage{mathrsfs}
\usepackage{MnSymbol}
\usepackage{color,soul}
\usepackage[text={4.5in,7.2in},centering]{geometry}
\usepackage{gensymb}
\usepackage{caption}
\theoremstyle{definition}
\newcounter{dummy} 
\numberwithin{dummy}{section}
\newtheorem{thm}[dummy]{Theorem}

\newtheorem{defin}[dummy]{Definition}

\newtheorem{cor}[dummy]{Corollary}
\newtheorem{prop}[dummy]{Proposition}

\DeclareMathAlphabet\mathbb{U}{msb}{m}{n}

\let\oldhat\hat                

\renewcommand{\hat}[1]{\oldhat{\mathbf{#1}}}

\usepackage{calc}
\usepackage{tikz}

\usetikzlibrary{decorations.markings}
\usetikzlibrary{positioning,automata}
\usetikzlibrary{shapes}
\usetikzlibrary{arrows}
\usepackage{verbatim}
\tikzset{
	>=stealth',
	vertex/.style={
		circle,
		draw=#1,
		fill=#1,
		inner sep = 2pt,
		outer sep = 0pt,
		text centered},
	edge/.style={
		-,
		thin,
		draw=black,
		},
	text style/.style={
		sloped,
		text=black,
		font=\normalsize,
		above}
	}

\tikzset{every loop/.style={min distance=10mm,in=45,out=135,looseness=0}}

\title{Edge Cut Domination, Irredundance, and Independence in Graphs }

\author{Todd Fenstermacher, Stephen Hedetniemi, Renu Laskar \\ Clemson University}

\begin{document}
	
	\begin{abstract}  An edge dominating set $F$ of a graph $G=(V,E)$ is an \textit{edge cut dominating set} if the subgraph $\langle V,G-F \rangle$ is disconnected. The \textit{edge cut domination number} $\gamma_{ct}(G)$ of $G$ is the minimum cardinality of an edge cut dominating set of $G.$ In this paper we study the edge cut domination number and investigate its relationships with other parameters of graphs. We also introduce the properties edge cut irredundance and edge cut independence.
		
	\end{abstract}
	
	\maketitle
	
	
\section {Introduction}

Let $G=(V,E)$ be a graph of order $n=|V|$ and size $m=|E|.$ Here we often take $G$ to be a connected simple graph. The \textit{open neighborhood} of a vertex $v \in V$ is the set $N(v) = \{u \text{ }|\text{ } uv \in E\},$ while the \textit{closed neighborhood} of $v$ is the set $N[v] = N(v) \bigcup \{v\}.$ Similarly, the \textit{closed neighborhood} of a set $S \subseteq V$ is $N[S] = \bigcup_{v \in S}N[v].$

A \textit{dominating set} is a set $S\subseteq V$ for which $N[S] = V.$ The \textit{domination number} $\gamma(G)$ equals the minimum cardinality over all dominating sets in $G,$ and a dominating set of cardinality $\gamma(G)$ is called a $\gamma$-set. A dominating set $S$ is called a \textit{split dominating set} if the induced subgraph $\langle V-S\rangle$ is either disconnected or $K_1.$ The \textit{split domination number} $\gamma_s(G)$ is the minimum cardinality of a split dominating set. This parameter was introduced by Kulli and Janakiram in 1997 \cite{Kulli}.

Similar parameters can be defined in terms of a set of edges. A set $F\subseteq E$ is called an \textit{edge dominating set} if every edge not in $F$ is adjacent to an edge in $F,$ that is, has a vertex in common with an edge in $F.$ The \textit{edge domination number} $\gamma'(G)$ is the minimum cardinality over all edge dominating sets of $G,$ and an edge dominating set of cardinality $\gamma'(G)$ is called a $\gamma'$-set. An edge dominating set $F$ is called an \textit{edge cut dominating set} if the subgraph $\langle V,E-F \rangle$ is disconnected. The \textit{edge cut domination number} $\gamma_{ct}(G)$ is the minimum cardinality over all edge cut dominating sets of $G,$ and an edge cut dominating set of cardinality $\gamma_{ct}(G)$ is called a $\gamma_{ct}$-set.

In 2001, Neeralagi and Nayak first introduced the edge cut domination number \cite{Nayak}; however, they named this parameter the \textit{split edge domination number}. We adopt the term edge cut domination number to indicate more clearly that a $\gamma_{ct}$-set is an edge dominating set containing an edge cut, and has nothing to do with a possible operation of splitting edges.

The \textit{edge connectivity} of a connected graph $G,$ denoted $\lambda(G),$ equals the minimum cardinality of a set of edges $F\subseteq E$ such that $\langle V,E-F \rangle$ is disconnected.  Such a set is called a $\lambda$-set, or a \textit{minimum edge cut.} We note that for a connected graph $G,$ both a $\gamma'$-set and $\lambda$-set exist. Thus, a $\gamma_{ct}$-set exists for any (connected) graph $G.$ \\

\section{Values and Bounds}

In this section we establish the value of the edge cut domination number for various classes of graphs, and we establish a variety of inequalities between this parameter and other known parameters of graphs. The following inequalities are consequences of the definitions of the given parameters, and are stated without proof.

\begin{prop}
	For any connected graph $G,$
	\[  (i) \text{ } \gamma'(G) \leq \gamma_{ct}(G) \hspace{2ex} \text{and} \hspace{2ex} (ii) \text{ } \lambda(G) \leq \gamma_{ct}(G). \]
\end{prop}	

Most of the following statements were previously observed by Neeralagi and Nayak \cite{Nayak}, although without proof or comment. Here we provide proofs of these statements. 

\begin{prop} \text{ } \\
(i) For the complete graph $K_n$ of order $n,$ $\gamma_{ct}(K_n) = n-1.$
 \\  (ii) For the cycle $C_n$ of order $n \geq 4,$ $\gamma_{ct}(C_n) = \lceil \frac{n}{3} \rceil .$ \\ (iii) For the wheel $W_n$ of order $n+1,$ $\gamma_{ct}(W_n) = \lceil \frac{n-4}{3}\rceil + 3.$ \\ (iv) For the complete bipartite graph $K_{m,n},$ with $m\geq n,$ $\gamma_{ct}(K_{m,n}) = n.$ \\ 
 (v) For any tree $T,$ $\gamma_{ct}(T) = \gamma'(T).$ \\
 (vi) For the path $P_n$ of order $n,$ $\gamma_{ct}(P_n) = \lceil \frac{n-1}{3}\rceil.$ 
\end{prop}

\begin{proof}
\textit{(i)} We have $\lambda(K_n) = n-1 \leq \gamma_{ct}(K_n).$ Note that the set of all edges incident to a given vertex is an edge cut dominating set of cardinality $n-1.$ Hence $\gamma_{ct}(K_n) = n-1.$ 

\textit{(ii)} For $n \geq 4,$ it is clear that $\lambda(C_n) = 2.$ Moreover, $\gamma'(C_n) = \lceil \frac{n}{3} \rceil \geq 2.$ It follows that $\gamma_{ct}(C_n) = \gamma'(C_n) = \lceil \frac{n}{3} \rceil.$ 

\textit{(iii)} For a given vertex of degree $3$, select the three edges incident to this vertex so that the resulting subgraph is disconnected. Note that dominating the remaining edges is equivalent to dominating $P_{n-3}.$ Since $\gamma'(P_{n-3}) = \lceil \frac{n-4}{3} \rceil,$ we have an edge cut dominating set of cardinality $\lceil \frac{n-4}{3} \rceil + 3.$ Moreover, such a set is a minimum cardinality edge cut dominating set. 

\textit{(iv)} The set of all edges incident to a given vertex in the partition of $K_{m,n}$ with $m$ vertices is an edge cut dominating set of cardinality $n.$ Since $\lambda(K_{m,n}) = n,$ it follows that $\gamma_{ct}(K_{m,n}) = n.$

\textit{(v)} Note that every edge of $T$ is a cut edge. Hence every $\gamma'$-set disconnects $T.$ It follows that $\gamma_{ct}(T) = \gamma'(T).$ 

\textit{(vi)} Since $P_n$ is a tree, we have $\gamma_{ct}(T) = \gamma'(T) = \lceil\frac{n-1}{3}\rceil.$

\end{proof}

The \textit{edge covering number} $\alpha_1(G)$ is the minimum cardinality of a set $F$ of edges such that every vertex is incident with at least one edge in $F.$ The \textit{matching number} $\beta_1(G)$ is the maximum cardinality over all independent edge sets.

\begin{prop}
For any connected graph G with size $m>1,$ \[ \gamma_{ct}(G) \leq m-\beta_1(G).  \]	
\end{prop}
	
\begin{proof}
	Let $F$ an independent set of edges in $G,$ that is no two edges in $F$ have a vertex in common. Now consider the complement of $F,$ $E-F.$ Since $F$ is an independent set of edges, we see that the removal of the edge set $E-F$ from $G$ disconnects $G.$ Moreover, $E-F$ is an edge dominating set. For if not, then there exists some edge $e \in F$ such that $e$ is not adjacent to any edge in $E-F.$ But of course $e$ is not adjacent to any edge in $F.$ But this contradicts the fact $G$ is a connected graph. Hence $E-F$ is an edge cut dominating set. Hence $\gamma'_{ct}(G) \leq |E - F| = m - |F|.$ 
	
	Since $\beta_1(G)$ is the maximum cardinality over all independent egde sets, it follows that $\gamma'_{ct}(G) \leq m - \beta_1(G).$ 
\end{proof}

The following corollary follows from the fact that for a connected graph $G=(V,E),$ $\alpha_1(G) + \beta_1(G) = |V|.$
\begin{cor}
	For any tree $T$ of order $n,$ 
	\[ \gamma_{ct}(G) \leq n - \beta_1(G) - 1 = \alpha_1(G) - 1.	\]
\end{cor}

Given the above lower and upper bounds for $\gamma_{ct}(G),$ it is of interest to determine for which classes of graphs any of the three following expressions hold:
\[	(i) \text{ } \gamma'(G) = \gamma_{ct}(G),  \hspace{3ex} (ii) \text{ } \lambda(G) = \gamma_{ct}(G), \hspace{3ex} (iii) \text{ } \gamma_{ct}(G) \leq \alpha_1(G) - 1. \]

Here we look briefly at some classes of graphs where $\gamma'(G) = \gamma_{ct}(G).$ There are of course some trivial cases. For example, if there exists a $\gamma'$-set which contains a cut edge of $G,$ then $\gamma'(G) = \gamma_{ct}(G).$ Note however, that if $\gamma'(G) = \gamma_{ct}(G)$ it is not necessarily the case that there exists a $\gamma'$-set containing a cut edge. For example, $\gamma'(K_{m,n}) = \gamma_{ct}(K_{m,n}),$ but $K_{m,n}$ contains no cut edge. There are of course an infinite number of graphs with the property that there exists a $\gamma'$-set containing a cut edge. It is clear than any tree has this property. More generally, if $G$ is a connected graph containing three adjacent cut edges, then every $\gamma'$-set of $G$ contains a cut edge. For in order to dominate $G,$ an edge set must contain at least one of the three adjacent cut edges. 

The following proposition presents another infinite class of graphs which have a $\gamma'$-set containing a cut edge.

\begin{prop}
	For a graph $G$ defined by $K_n$ and $K_m$ ($m,n > 2$) connected by a path of length $1$ or $2,$ $\gamma_{ct}(G) = \gamma'(G)$ if and only if $m$ or $n$ is even. 
\end{prop}

\begin{proof}
	First recall that $\gamma_{ct}(K_n) = n-1,$ but $\gamma'(K_n) = \lfloor n/2 \rfloor.$  So a $\gamma_{ct}$-set must contain a cut edge from the path connecting $K_n$ and $K_m.$ 
	
	Consider the case when $K_m$ and $K_n$ are connected by a single edge, say $e.$ Note that $e$ is a cut edge. Now there are two distinct edge sets to consider. Let $E_1$ be an edge dominating set of minimal cardinality which does not contain $e.$ Let $E_2$ be an edge dominating set of minimal cardinality which does contain $e.$ Given that $\gamma'(K_n) = \lfloor n/2 \rfloor,$ we see that
	\[ |E_1| = \left\lfloor \frac{m}{2} \right\rfloor + \left\lfloor \frac{n}{2} \right\rfloor  \hspace{4ex} |E_2| = 1 + \left\lfloor \frac{m-1}{2} \right \rfloor + \left\lfloor \frac{n-1}{2} \right \rfloor \]
	
	Now we make the observation that $\gamma'(G) = \min\{|E_1|, |E_2|\}$ and $\gamma_{ct}(G) = |E_2|.$ Hence we have $\gamma'(G) = \gamma_{ct}(G)$ if and only if $|E_2| \leq |E_1|.$ It is straightforward to check that $|E_2| \leq |E_1|$ if and only if at least one of $m$ and $n$ is even.

	Next consider the case when $K_m$ and $K_n$ are connected by a path of length two. Let the two edges of the path be denoted by $d$ and $e,$ and say $d$ is incident to a vertex in $K_m$ and $e$ is incident to a vertex in $K_n.$ In this case there are four sets of edges to consider. Let $E_1, E_2, E_3, E_4$ each be edge dominating sets of minimum cardinality which also satisfy the following conditions: 
	\[ d,e \notin E_1,  \hspace{2ex} d \in E_2, e \notin E_2,   \hspace{2ex} d \notin E_3, e \in E_3, \hspace{2ex} d,e \in E_4. \]
	
	Again using the observation that $\gamma'(K_n) = \lfloor n/2 \rfloor$ we have that 
	\[ |E_1| =   \left\lfloor \frac{m}{2} \right \rfloor + \left\lfloor \frac{n}{2} \right \rfloor \hspace{2ex} |E_2| = 1 + \left\lfloor \frac{m-1}{2} \right \rfloor + \left\lfloor \frac{n}{2} \right \rfloor \] \[ |E_3| = 1 + \left\lfloor \frac{m}{2} \right \rfloor + \left\lfloor \frac{n-1}{2} \right \rfloor \hspace{2ex} |E_4| = 2 + \left\lfloor \frac{m-1}{2} \right \rfloor + \left\lfloor \frac{n-1}{2} \right \rfloor \]
	
	It is clear that $\gamma'(G) = \min\{|E_1|, |E_2|, |E_3|, |E_4| \}$ and also that $\gamma_{ct} = \min\{|E_2|, |E_3|, |E_4|\}.$ Hence $\gamma'(G) = \gamma_{ct}(G)$ if and only if $|E_i| \leq |E_1|$ for some $i \in \{2,3,4\}.$ Again a straightforward check shows that this occurs if and only if at least one of $m$ and $n$ is even. 	
\end{proof}

\section{Introduction of New Parameters}

In the mid 1970s Cockayne and Hedetniemi \cite{Cockayne1, Cockayne2} noted the following chain of inequalities for any graph $G$:
\[  \text{ir}(G) \leq \gamma(G) \leq i(G) \leq \beta(G) \leq \Gamma(G) \leq \text{IR}(G). \] 

Here $\Gamma(G),$ the \textit{upper domination number}, is the maximum cardinality taken over all minimal dominating sets of $G.$ The \textit{independent domination number} and \textit{independence number}, $i(G)$ and $\beta(G),$ are respectively the minimum and maximum cardinalities taken over all maximal sets of independent vertices of $G.$ Similarly, $\text{ir}(G)$ and $\text{IR}(G),$ the \textit{lower} and \textit{upper irredundance numbers}, are respectively the minimum and maximum cardinalities taken over all maximal irredundant sets of vertices of $G.$ 

The introduction and study of new parameters often involves a similar chain of inequalities as the one above. Indeed, this inequality chain has been instrumental in the research of many parameters. In what follows we establish such a chain relating parameters which we define corresponding to edge cut domination, edge cut irredundance, and edge cut independence.

\begin{defin}
	An edge cut dominating set $F$ is a \textit{minimal edge cut dominating set} if for any edge $e$ in $F$ either
	\begin{enumerate} 
		\item $F-\{e\}$ is not an edge dominating set, or 
		\item $F-\{e\}$ is not an edge cut.
	\end{enumerate}
\end{defin} 

\begin{defin} Let $G = (V,E)$ be a graph. Then 
	\[ \Gamma_{ct}(G) = \max\{ |F| : F \text{ is a minimal edge cut dominating set}\}. \]
\end{defin}

Since a minimal edge cut dominating set is first and foremost an edge cut dominating set, it is clear that $\gamma_{ct}(G) \leq \Gamma_{ct}(G).$ 


\begin{defin} Let $G=(V,E)$ be a graph and $F \subseteq E.$ Then
	an edge $e \in F$ has a \textit{private neighbor} with respect to $F$ if either
	\begin{enumerate}
		\item $e$ is an independent edge in $F,$ or 
		\item $\exists$ $e' \notin F$ such that $e'$ is adjacent to $e$ and no other edges of $F.$	
	\end{enumerate} 
\end{defin} 

\begin{defin} Let $G = (V,E)$ be a graph, and $F \subseteq E.$ Then
	an edge $e \in F $ is \textit{irredundant} if $e$ has a private neighbor with respect to $F.$ If each edge in $F$ is irredundant, then we say $F$ is irredundant.
\end{defin}

\begin{defin} 
	Let $G=(V,E)$ be a graph. Then $F \subseteq E$ is an \textit{edge cut irredundant set} if for every edge $e \in F$ either
	\begin{enumerate}
		\item $e$ is irredundant, or 
		\item $F-\{e\}$ is not an edge cut.
	\end{enumerate} 
\end{defin}

\begin{defin} 
	An edge cut irredundant set $F$ is called \textit{maximal} if $F\bigcup \{e\}$ is not edge cut irredundant for every $e \in E- F.$	
\end{defin}

\begin{defin} Let $G=(V,E)$ be a graph. Then
 \begin{itemize}
 	\item $\text{ir}_{ct} = \min\{|F| : F \text{ is a maximal edge cut irredundant set}\},$ and 
 	\item $\text{IR}_{ct}= \max\{ |F| : F \text{ is a maximal edge cut irredundant set}\}.$ 
 \end{itemize}
\end{defin}

\begin{prop}\label{min=max}
	A minimal edge cut dominating set $F$ is a maximal edge cut irredundant set.
\end{prop}

\begin{proof}
	Let $F$ is a minimal edge cut dominating set. Now for any edge $e \in F,$ either $F - \{e\}$ is not an edge dominating set, which means that $e$ has a private neighbor with respect to $F,$ i.e., $e$ is irredundant; or $F-\{e\}$ is not an edge cut. So $F$ is an edge cut irredundant set. Now let $e \in E-F$ and consider $F\bigcup\{e\}.$ Since $F$ is an edge dominating set, we know that $F \bigcup \{e\}$ is also an edge dominating set. But this implies that $e$ has no private neighbor with respect to $F \bigcup \{e\}.$ For $e$ cannot be independent in $F \subseteq F\bigcup\{e\},$ and every edge adjacent to $e$ must also be adjacent to some edge in $F.$ It follows that $e$ is not irredundant in $F\bigcup \{e\}.$ Hence $F\cup\{e\}$ is not an edge cut irredundant set for any $e \in E-F.$ Thus $F$ is a maximal edge cut irredundant set.  
\end{proof}

Note: A maximal edge cut irredundant set is not necessarily an edge cut dominating set. This can be seen in the following figure.\\

\begin{center}
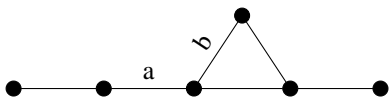

	\begin{tikzpicture}
	\node[vertex] (0) {};	
	\node[vertex, right= 1cm of 0] (1) {}
	edge[edge] node {} (0);
	\node[vertex, right= 1cm of 1] (2) {}
	edge[edge] node[above] {a} (1);
	\node[vertex, above right=.83cm and .5cm of 2] (3) {}
	edge[edge] node[text style, above] {b} (2);
	\node[vertex, below right= .83cm and .5cm of 3] (4) {}
	edge[edge] node {} (3)
	edge[edge] node {} (2);
	\node[vertex, right= 1cm of 4] {}
	edge[edge] node {} (4);
	\end{tikzpicture}
	\captionof{figure}{The edge set $\{a,b\}$ is a maximal edge cut irredundant set but not an edge cut dominating set.}
\end{center}

\begin{defin} Let $G=(V,E)$ be a graph, and $F \subseteq E.$ Then
	F is an \textit{edge cut independent set} if for every edge $e \in F$ either
	\begin{enumerate}
		\item $e$ is independent in $F,$ or 
		\item $F-\{e\}$ is not an edge cut.
	\end{enumerate} 
\end{defin}

\begin{defin}
	An edge cut independent set $F$ is called \textit{maximal} if $F\bigcup \{e\}$ is not an edge cut independent set for every edge $e \in E-F.$
\end{defin} 

\begin{defin} Let $G=(V,E)$ be a graph. Then
	\begin{itemize}
		\item $i_{ct}(G) = \min\{ |F| : F \text{ is maximal edge cut independent set}\},$ and 
		\item $\beta_{ct}(G) = \max\{ |F| : F \text{ is maximal edge cut independent set}\}.$
		
	\end{itemize} 
\end{defin}



\begin{prop}\label{prop1}
	A maximal edge cut independent set is a minimal edge cut dominating set. 
\end{prop}
 
\begin{proof}
	Let $F$ be a maximal edge cut independent set. We first show that $F$ is an edge dominating set. For suppose this is not the case. Then there exists some $e \in E-F$ such that $e$ is not adjacent to any edge in $F.$ But this implies that $e$ is independent in $F\bigcup \{e\},$ which contradicts the maximality of $F.$ Hence $F$ is an edge (cut) dominating set. Now we show that $F$ is a minimal edge cut dominating set. Let $e \in F,$ then either $e$ is independent in $F,$ or $F-\{e\}$ is not an edge cut. If the latter is true, we are done. On the other hand, if $e$ is independent in $F,$ then $e$ is not adjacent to any edge in $F,$ which implies that $F-\{e\}$ is not an edge dominating set. Therefore, $F$ is a minimal edge cut dominating set. 
\end{proof}

Note: A minimal edge cut dominating set is not necessarily a maximal edge cut independent set. This is shown in the following figure.

\begin{center}
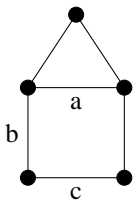

	\begin{tikzpicture}
	\node[vertex] (0) {};	
	\node[vertex, below right= .83cm and .5cm of 0] (1) {}
	edge[edge] node {} (0);
	\node[vertex, below left= .83cm and .5cm of 0] (2) {}
	edge[edge] node {} (0)
	edge[edge] node[below] {a} (1);
	\node[vertex, below = 1cm of 1] (3) {}
	edge[edge] node {} (1);
	\node[vertex, below = 1cm of 2] (4) {}
	edge[edge] node[left] {b} (2)
	edge[edge] node[below] {c} (3);
	\end{tikzpicture}
	\captionof{figure}{The edge set $\{a,b,c\}$ is a minimal edge cut dominating set but not an edge cut independent set.}
\end{center}

\begin{thm}
	For any connected graph $G,$ \[  \text{ir}_{ct}(G) \leq \gamma_{ct}(G) \leq i_{ct}(G) \leq \beta_{ct}(G) \leq \Gamma_{ct}(G) \leq \text{IR}_{ct}(G).  \]
\end{thm}

\begin{proof}
	From Proposition $\ref{min=max}$ and the following note, we see that \[ \text{ir}_{ct}(G)  \leq \gamma_{ct}(G) \leq \Gamma_{ct}(G) \leq \text{IR}_{ct}(G). \]
	
	From Proposition \ref{prop1} and the following note, we also see that \[ \gamma_{ct}(G) \leq i_{ct}(G) \leq \beta_{ct}(G) \leq \Gamma_{ct}(G).\]
	
	Hence the desired inequality chain holds. 
\end{proof}

\centerline{}

\end{document}